\documentclass[10pt]{amsart}

\usepackage{amsmath, amsthm, amssymb}

\allowdisplaybreaks[2]

\newcommand{\R}{\mathbb{R}}
\newcommand{\N}{\mathbb{N}}

\newcommand{\SL}{\mathcal{L}}

\renewcommand{\S}{\mathbb{S}}
\newcommand{\IP}[2]{\left<#1,#2\right>}
\newcommand{\vn}[1]{\lVert#1\rVert}
\newcommand{\rd}[2]{\frac{d#1}{d#2}}
\newcommand{\pd}[2]{\frac{\partial#1}{\partial#2}}
\newcommand{\Ko}{K_{\text{$\mspace{-1mu}o\mspace{-1mu}s\mspace{-1mu}c$}}}

\newcommand{\kav}{\overline{k}}

\newtheorem{thm}{Theorem}[section]
\newtheorem*{thm*}{Theorem}
\newtheorem{prop}[thm]{Proposition}
\newtheorem{lem}[thm]{Lemma}
\newtheorem{cor}[thm]{Corollary}

\theoremstyle{definition}
\newtheorem{rmk}[thm]{Remark}
\begin{document}

\title{On the curve diffusion flow of closed plane curves}
\author{Glen Wheeler}

\thanks{Financial support from the Alexander-von-Humboldt Stiftung is gratefully acknowledged}
\address{Otto-von-Guericke-Universit\"at\\
Postfach 4120\\
D-39016 Magdeburg}
\email{wheeler@ovgu.de}
\subjclass[2000]{53C44 \and 58J35} 

\begin{abstract}
In this paper we consider the steepest descent $H^{-1}$-gradient flow of the length functional for immersed plane curves,
known as the curve diffusion flow.
It is known that under this flow there exist both initially immersed curves which develop at least one singularity in finite
time and initially embedded curves which self-intersect in finite time.
We prove that under the flow closed curves with initial data close to a round circle in the sense of normalised $L^2$
oscillation of curvature exist for all time and converge exponentially fast to a round circle.
This implies that for a sufficiently large `waiting time' the evolving curves are strictly convex.
We provide an optimal estimate for this waiting time, which gives a quantified feeling for the magnitude to which the maximum
principle fails.
We are also able to control the maximum of the multiplicity of the curve along the evolution.
A corollary of this estimate is that initially embedded curves satisfying the hypotheses of the global existence theorem
remain embedded.
Finally, as an application we obtain a rigidity statement for closed planar curves with winding number one.

\end{abstract}
\maketitle

\section{Introduction}

Suppose $\gamma:\S^1\rightarrow\R^2$ is an immersed closed plane curve of period $P$ and consider the energy
\[
L(\gamma) = \int_0^P |\gamma_u|\, du,
\]
where $\gamma_u = \partial_u\gamma$.
We wish to deform $\gamma$ towards a minimiser of $L$, and for this purpose we shall consider the steepest descent gradient
flow of $L$ in $H^{-1}$.
There are some advantages in choosing $H^{-1}$ instead of $L^2$.
One is that for any initial curve the signed area is constant under the flow, which implies that if the signed area of the
initial curve is non-zero, then the flow is never asymptotic to a lower dimensional subset of $\R^2$.

The Euler-Lagrange operator of $L$ in $H^{-1}$ is
\[
\text{grad}_{H^{-1}}L(\gamma) = k_{ss},
\]
where $k=\IP{\gamma_{ss}}{\nu}$ is the curvature of $\gamma$, $\nu$ a unit normal vector field on $\gamma$, and $s$ denotes
arc-length.
The curve diffusion flow is the one-parameter family of immersed curves
$\gamma:\S^1\times[0,T)\rightarrow\R^2$ with normal velocity equal to
$-\text{grad}_{H^{-1}}(L(\gamma))$, that is
\begin{equation}
\partial^\perp_t\!\gamma = -k_{ss}.
\label{CD}
\tag{CD}
\end{equation}
The curve diffusion flow is a degenerate system of quasilinear fourth order parabolic partial differential equations, and as
such it is not expected that a maximum or comparison principle holds.
Indeed, Giga and Ito \cite{GI98pinching} provided the first example of a simple, closed, strictly embedded planar curve
which develops a self-intersection in finite time under the flow.
They also gave \cite{GI99loss} the first example of a simple, closed, strictly convex planar curve which becomes
non-convex in finite time.
Furthermore, Elliot and Maier-Paape showed \cite{EM01losing} that the curve diffusion flow may drive an initial graph to
become non-graphical in finite time.
It was eventually shown by Blatt \cite{B10} that non-preservation of convexity and non-preservation of embeddedness is a basic
property of a large class of general higher order hypersurface flows.

It is also known (see Polden \cite{P96} for the first example and Escher-Ito \cite{EI05} for many others) that the curve
diffusion flow can from smooth immersed initial data develop finite time curvature singularities.
In contrast, our goal in this paper is to demonstrate a new class of initial data (generalising \cite[Theorem 6.1]{EG97})
which gives rise to an immortal solution converging exponentially fast to a simple round circle.

The curve diffusion flow has been considered for some time in the literature.
The first point to note is that for regular enough initial data $\gamma_0:\S^1\rightarrow\R^2$ there is a maximal
	$T\in(0,\infty]$ and corresponding solution $\gamma:\S^1\times[0,T)\rightarrow\R^2$ which satisfies \eqref{CD}.
Local existence, although technical and sometimes tricky, is by now standard---in this paper we state a version (Theorem
\ref{STE}) which is a combination of Elliot-Garcke \cite{EG97} and Dziuk-Kuwert-Sch\"atzle \cite[Theorem 3.1]{DKS02},
although similar results appeared earlier, see \cite{BDR84,CTNc96,EG97,CT94} for example.
It is also quite standard regardless: as mentioned, the evolution equation \eqref{CD} is a degenerate fourth-order
quasilinear parabolic system, and local existence can be obtained for example through the method of semigroups (Angenent
\cite{A90}, Amann \cite{A93,Abook,A05}, Escher-Meyer-Simonett \cite{EMS98,ES99}, and Lunardi \cite{Lbook} are good
references), the Nash-Moser inverse function theorem (see Hamilton \cite{H82nash,H82}, and Gage-Hamilton \cite{GH86}) or
through more classical methods such as can be found in Polden \cite{P96} and Huisken-Polden \cite{HP96} (see also Sharples
\cite{S04} and the books \cite{Ebook,EZbook,Fbook}).
The local existence theorem we use requires that the curvature of $\gamma$ lies in $L^2$.
One should note that there are local existence results which do not require any control of curvature, instead requiring
Lipschitz with small Lipschitz constant or slightly more regularity than $C^1$ for the initial data, see Koch-Lamm
\cite{KL11}, Escher-Mucha \cite{EM10surface}, and Asai \cite{A10smoothing} for example.

The analysis we present here is direct and geometric in nature, and should be compared with
\cite{BBW98,C91,DKS02,H82,KS01,KS02,KS04,P96,W09}.
It rests on the observation that the normalised oscillation of curvature
\[
\Ko\big(\gamma(\cdot,t)\big) = L\big(\gamma(\cdot,t)\big)\int_\gamma (k-\kav)^2 ds,
\]
where $\kav$ denotes the average of the curvature, is in many respects a natural `energy' for the flow.
The only stationary solutions of \eqref{CD} are lines and multiply covered circles, for which $\Ko = 0$.
Further, for arbitrary smooth initial data $\int_0^t \Ko\, d\tau \le L^4(\gamma(\cdot,0))/16\pi^2$ (see Lemma \ref{FO}), that
is, $\Ko \in L^1\big([0,T)\big)$.

We prove that if $\Ko$ is initially small and the isoperimetric ratio $I = L^2/4\pi A$ is initially close to
one, then they remain so.
This is enough to begin a `bootstrapping' style procedure, in which we use interpolation
inequalities as in \cite{DKS02} to obtain uniform bounds for all higher derivatives of curvature.
These observations and some extra arguments give the global existence result of this paper.

\begin{thm}
\label{GE1}
Suppose $\gamma_0:\S^1\rightarrow\R^2$ is a regular smooth immersed closed curve with
$A(\gamma_0)>0$ and
\begin{equation}
\label{Ewind1}
\int_{\gamma_0}k\, ds = 2\pi.
\end{equation}
There exists a constant $K^* > 0$ such that if
\begin{equation}
\label{Einitialconds}
\Ko(\gamma_0) < K^*,\quad\text{ and }\quad I(\gamma_0) < \exp\Big(\frac{K^*}{8\pi^2}\Big),
\end{equation}
then the curve diffusion flow $\gamma:\S^1\times[0,T)\rightarrow\R^2$ with $\gamma_0$ as initial data exists for all time and converges 
exponentially fast to a round circle with radius $\sqrt{\frac{A(0)}{\pi}}$.
\end{thm}

\begin{rmk}
\label{RMK1}
One advantage of our direct method is that we are able to easily find an allowable choice for the
constant $K^*$ above; in particular, one may select
\[
K^* = \frac{2\pi + 12\pi^2 - 4\pi\sqrt{3\pi}\sqrt{1+3\pi}}{3} \simeq \frac{1}{18}.
\]
\end{rmk}

\begin{rmk}
So long as $A(\gamma_0) \ne 0$, one may always guarantee $A(\gamma_0)>0$ by reversing the orientation of $\nu$, since \eqref{CD} is invariant under change of orientation.
\end{rmk}

\begin{rmk}
As can be seen from the proof of Proposition \ref{EE1}, the smallness condition
\eqref{Einitialconds} could be weakened to
\[
\Ko(\gamma_0)+8\pi^2\log\sqrt{I(\gamma_0)} \le 2K^*-\delta,
\]
for any $\delta>0$.  We do not expect this to be optimal, however.
At this time, it is not known if there exists any smooth plane curve satisfying \eqref{Ewind1} which
gives rise to a curve diffusion flow with finite maximal existence time.  Without at least one such
singular example, it is difficult to even conjecture on what an optimal form of
\eqref{Einitialconds} may be.
\end{rmk}

It is clear that Theorem \ref{GE1} implies 
$k(\cdot,t) \rightarrow \sqrt{\frac{\pi}{A(0)}}$, and so after a fixed time translation we have
\[
k(\cdot,t) \ge \sqrt{c} > 0
\]
for any $c \in (0,\frac{\pi}{A(0)})$ (cf. \cite[Lemma 5.5]{KS01} and \cite{W10} for the Willmore
flow and surface diffusion flow of surfaces respectively).  In other words, after some finite time
the curvature becomes positive and remains so.  This can be thought of as `eventual positivity', and
is reminiscent of the situation considered in \cite{FGG08,GG08,GG09}.  There, using very different
techniques, eventual local positivity and other related qualitative properties are observed for
biharmonic parabolic equations under certain conditions.  To further quantify the size of the
`waiting time', we present the following.

\begin{prop}
\label{PwtE}
Suppose $\gamma:\S^1\times[0,T)\rightarrow\R^2$ solves \eqref{CD} and satisfies the assumptions of Theorem \ref{GE1}.  Then
\[
\SL\big\{ t\in[0,\infty) : k(\cdot,t) \not> 0\big\}
 \le \Big(\frac{L(\gamma_0)}{2\pi}\Big)^4-\Big(\frac{A(\gamma_0)}{\pi}\Big)^2.
\]
\end{prop}
In the above, $k(\cdot,t) \not> 0$ means that there exists a $p$ such that $k(p,t) \le 0$.  This
estimate is optimal in the sense that the right hand side is zero for a simple circle.

It is not clear at all from Theorem \ref{GE1} if initially embedded curves remain so, nor even if we can control the maximum
of the multiplicity (the number of times the curve intersects itself in one point) of the evolving curve.
We do have good control of the oscillation of curvature however, and in the spirit of \cite[Theorem 6]{LY82} (see also the
monotonicity formula in \cite{S93} and appendix of \cite{KS04}) present the following theorem to address this issue.

\begin{thm}
\label{Tmult}
Suppose $\gamma:\S^1\rightarrow\R^2$ is a smooth immersed curve with winding number $\omega$ and let $m$ denote the maximum number of times $\gamma$ intersects itself in any one point; that is
\[
\int_\gamma k\, ds = 2\omega\pi\quad\text{and}\quad
m(\gamma) = \sup_{x\in\R^2} |\gamma^{-1}(x)|.
\]
Then
\[
\Ko(\gamma) \ge 16m^2 - 4\omega^2\pi^2.
\]
\end{thm}

When combined with Proposition \ref{EE1} we obtain the following.

\begin{cor}
\label{Cmult}
Any curve diffusion flow $\gamma:\S^1\times[0,T)\rightarrow\R^2$ with initial data
$\gamma_0:\S^1\rightarrow\R^2$ satisfying the assumptions of Theorem \ref{GE1} with
\[
K^* < 64-4\pi^2 \simeq 24.5
\]
remains embedded for all time.
\end{cor}

Note that in particular the allowable choice for $K^*$ given in Remark \ref{RMK1} is smaller than
$64-4\pi^2$.

Theorem \ref{GE1} gives a one-parameter family of smooth diffeomorphisms connecting the initial data
$\gamma_0$ with a round circle.  This implies the following rigidity result.

\begin{cor}
Let $\gamma:\S^1\rightarrow\R^2$ be a regular closed immersed curve satisfying the assumptions
of Theorem \ref{GE1}.  Then $\gamma$ is diffeomorphic to a round circle.
\end{cor}

This paper is organised as follows.
In Section 2 we fix our notation, state the local existence theorem, and prove some elementary Sobolev-Poincar\'e-Wirtinger inequalities.
Section 3 contains estimates for the curvature in $L^2$ and the isoperimetric ratio under various
assumptions, which forms the bulk of the work involved in proving Theorem \ref{GE1}.
The theorem itself and Proposition \ref{PwtE} are proved in Section 4.
We finish the paper by proving Theorem \ref{Tmult} and Corollary \ref{Cmult} in Section 5.

\section*{Acknowledgements}

The author thanks his colleagues for several useful discussions, in particular Hans-Christoph Grunau for reading an early
version of this paper.
The author would also like to thank Ernst Kuwert for helpful discussions at the Mathematisches Forschungsinstitut Oberwolfach
(MFO).
This work was completed under the financial support of the Alexander von Humboldt Stiftung at the
Otto-von-Guericke-Universit\"at Magdeburg.

\section{Preliminaries}

Suppose $\gamma:\R\rightarrow\R^2$ is a regular smooth immersed plane curve.
We say that $\gamma$ is periodic with period $P$ if there exists a vector $V\in\R^2$ and a positive $P$ such that for all $m\in\N$
\[
\gamma(u+P) = \gamma(u)+V,\quad\text{  and  }\quad\partial^{m}_u\gamma(u+P) = \partial_u^{m}\gamma(u).
\]
If $V = 0$ then $\gamma$ is closed.  In this case $\gamma$ is an immersed circle, $\gamma:\S^1\rightarrow\R^2$.
The length of $\gamma$ is
\[
L(\gamma) = \int_0^P |\gamma_u|\, du,
\]
and the signed enclosed area is
\begin{equation}
A(\gamma) = -\frac{1}{2}\int_0^P \IP{\gamma}{\nu}|\gamma_u|\,du,
\label{AF}
\end{equation}
where $\nu$ is a unit normal vector field on $\gamma$.
Throughout the paper we keep $\gamma$ parametrised by arc-length $s$, where $ds = |\gamma_u| du$.
Integrals over $\gamma$ are to be interpreted as integrals over the interval of periodicity.

Consider the one-parameter family of immersed curves $\gamma:\S^1\times[0,T)\rightarrow\R^2$ with normal velocity equal to $-\text{grad}_{H^{-1}}(L(\gamma))$, that is
\begin{equation*}
\partial^\perp_t\!\gamma = -k_{ss}.
\tag{CD}
\end{equation*}

The following theorem is standard.
The uniqueness below is understood modulo the natural group of invariances enjoyed by \eqref{CD}:
rotations, translations, changes of orientation, and so on, as is customary for geometric flows.

\begin{thm}[Local existence]
\label{STE}
Suppose $\gamma_0:\R\rightarrow\R^2$ is a periodic regular curve parametrised by arc-length and of class $C^1\cap W^{2,2}$ with $\vn{k}_2 < \infty$.
Then there exists a $T\in(0,\infty]$ and a unique one-parameter family of immersions
$\gamma:\R\times[0,T)\rightarrow\R^2$ parametrised by arc-length such that
\begin{enumerate}
\item[(i)]
$\gamma(0,\cdot) = \gamma_0$;
\item[(ii)]
$\partial^\perp_t\!\gamma = -k_{ss}$;
\item[(iii)]
$\gamma(\cdot,t)$ is of class $C^\infty$ and periodic of period $L(\gamma(\cdot,t))$ for every $t\in(0,T)$;
\item[(iv)]
$T$ is maximal.
\end{enumerate}
\end{thm}

Theorem \ref{STE} justifies the use of smooth calculations in the derivation of our estimates.  When
we use the expression ``$\gamma:\S^1\times[0,T)\rightarrow\R^2$ solves \eqref{CD}'' we are invoking
Theorem \ref{STE} in the special case where the initial data is assumed to be closed, but not necessarily embedded..

We will need the following elementary Sobolev-Poincar\'e-Wirtinger inequalities.

\begin{lem}
\label{WI}
Suppose $f:\R\rightarrow\R$ is absolutely continuous and periodic with period $P$.  Then if $\int_0^P f\, dx = 0$ we have 
\[
\int_0^P f^2 dx \le \frac{P^2}{4\pi^2}\int_0^P |f_x|^2 dx,
\]
with equality if and only if $f(x) = a\sin(2x\pi/P+b)$.
\end{lem}
\begin{proof}
Expand $f$ as a Fourier series and then use Parseval's identity.
\end{proof}

\begin{cor}
\label{LU}
Under the assumptions of Lemma \ref{WI},
\[
\vn{f}_\infty^2 \le \frac{P}{2\pi}\vn{f_x}_2^2.
\]
\end{cor}
\begin{proof}
As $f$ has zero average, there exist $p_1,p_2$ such that $f(p_1) = f(p_2) = 0$ and $0 \le p_1 < p_2 < P$.  Thus, since $f$ is
absolutely continuous and periodic,
\[
f^2(x) = \int_{p_1}^x f(u)f_x(u)\, du - \int_x^{p_2} f(u)f_x(u)\, du.
\]
Therefore
\[
f^2(x) \le \int_{p_1}^{p_2} |f(u)f_x(u)|\, du \le \int_0^P |f(u)f_x(u)|\, du.
\]
Now H\"older's inequality and Lemma \ref{WI} above implies
\[
\vn{f}_\infty^2 \le \vn{f}_2\vn{f_x}_2 \le \frac{P}{2\pi}\vn{f_x}_2^2,
\]
as required.
\end{proof}

As most of our analysis is based on integral estimates, it is efficient to first compute the derivative of an integral along the flow in general.

\begin{lem}
Suppose $\gamma:\S^1\times[0,T)\rightarrow\R^2$ solves \eqref{CD}, and $f:\S^1\times[0,T)\rightarrow\R$ is a periodic function with the same period as $\gamma$.
Then
\[
\rd{}{t}\int_\gamma f ds
 =  \int_\gamma f_t + fkk_{ss} - f_s(\partial^\top_t\!\gamma) ds.
\]
\label{DI}
\end{lem}
\begin{proof}
First note that $\tau = \gamma_u/|\gamma_u| = \gamma_s$ is a unit tangent vector field along $\gamma$.  We compute
\begin{align*}
\pd{}{t}|\gamma_u|^2
 &= 2\IP{\gamma_{ut}}{\gamma_u}
  = 2\IP{\partial_u\big((\partial_t^\perp\!\gamma)\nu + (\partial_t^\top\!\gamma)\tau\big)}{|\gamma_u|\tau}
\\
 &= 2\partial_t^\perp\!\gamma\IP{\partial_u\nu}{|\gamma_u|\tau}
   + 2|\gamma_u|\partial_u\partial_t^\top\!\gamma
\\
 &= -2\partial_t^\perp\!\gamma|\gamma_u|\IP{k|\gamma_u|\tau}{\tau}
   + 2|\gamma_u|\partial_u\partial_t^\top\!\gamma
\\
 &= -2k\partial_t^\perp\!\gamma|\gamma_u|^2
   + 2(\partial_s\partial_t^\top\!\gamma)|\gamma_u|^2.
\end{align*}
Therefore
\[
\pd{}{t}ds 
 = kk_{ss} ds
   + (\partial_s\partial_t^\top\!\gamma) ds.
\]
Using this we differentiate the integral to find
\begin{align*}
\rd{}{t}\int_\gamma f ds
 &=  \rd{}{t}\int_0^P f |\gamma_u| du
\\
 &=  \int_0^P f_t ds
   + \int_0^P f(kk_{ss} + \partial_s\partial^\top_t\!\gamma) ds
   + f(P,t) |\gamma_u(P)| P'
\\
 &=  \int_\gamma f_t + fkk_{ss} - f_s(\partial^\top_t\!\gamma) ds.
\end{align*}
We obtained the last equality using integration by parts and the periodicity of $\gamma$ with the identity
\[
\partial_t^\top\!\gamma(u,t)
-
\partial_t^\top\!\gamma(u+P(t),t)
=
|\gamma_u(u,t)|P'(t),
\]
which in turn follows from the definition of $\gamma$.  (Note in particular that the tangential velocity $\partial_t^\top\!\gamma$ is not periodic.)
\end{proof}

\section{Curvature estimates in $L^2$ and the isoperimetric ratio}

The evolution equation \eqref{CD} is particularly natural as solutions decrease in length while keeping enclosed area fixed.
This is only necessarily true for curves immersed in $\R^2$, and in fact this is the chief reason why we consider plane curves as opposed to curves in $\R^n$ or immersed in a manifold.

\begin{lem}
\label{LA}
Suppose $\gamma:\S^1\times[0,T)\rightarrow\R^2$ solves \eqref{CD}.
Then
\[
\rd{}{t}L = -\int_{\gamma} k_s^2 ds,\qquad \text{and}\qquad \rd{}{t}A = 0.
\]
In particular, the isoperimetric ratio decreases in absolute value with velocity
\[
\rd{}{t}I = -\frac{2I}{L}\int_{\gamma} k_s^2 ds.
\]
\end{lem}
\begin{proof}
Lemma \ref{DI} with $f \equiv 1$ gives
\[
\rd{}{t}L
 = \int_{\gamma} kk_{ss} ds
 = -\int_{\gamma} k_s^2 ds,
\]
where we used integration by parts and the periodicity of the curve.
For the area, we first note that
\begin{equation}
\tau_s = k\nu,\quad \nu_s = -k\tau,\quad \nu_t = k_{sss}\tau - (\partial^\top_t\!\gamma)k\tau.
\label{SF}
\end{equation}
The first two relations are immediate from differentiating $\IP{\tau}{\nu} = 0$ and using the definition of the curvature.
For the third, we first compute the commutator of the arc-length and time derivatives:
\begin{align}
\partial_{ts}
 &= \partial_t \big(|\gamma_u|^{-1} \partial_u\big)
 = |\gamma_u|^{-1} \partial_{tu} - |\gamma_u|^{-2} \big(\partial_t|\gamma_u|\big) \partial_{u} 
\notag\\
 &= \partial_{st} 
  + k\big(\partial_t^\perp\!\gamma\big)\partial_s
  - \big(\partial_s\partial_t^\top\!\gamma\big)\partial_s
\notag\\ \label{COM}
 &= \partial_{st} 
  - kk_{ss}\partial_s
  - \big(\partial_s\partial_t^\top\!\gamma\big)\partial_s.
\end{align}
Using this and the first two equalities in \eqref{SF} we compute the evolution of the unit tangent vector field $\tau$.
\begin{align}
\partial_t\tau
 &= \partial_{ts}\gamma
\notag
\\
 &= \partial_s\big( (\partial^\perp_t\!\gamma)\nu + (\partial^\top_t\!\gamma)\tau \big)
  - kk_{ss}\gamma_s
  - \big(\partial_s\partial_t^\top\!\gamma\big)\gamma_s
\notag
\\
 &= -k_{sss}\nu - k_{ss}\nu_s
  + (\partial^\top_t\!\gamma)\gamma_{ss}
  - kk_{ss}\gamma_s
\notag
\\
 &= -k_{sss}\nu
  + (\partial^\top_t\!\gamma)\gamma_{ss}.
\label{ET}
\end{align}
Noting that $|\nu|^2 = 1$ implies $\nu_t$ has no normal component, we obtain the final equality in \eqref{SF} by differentiating $\IP{\nu}{\tau}$
\begin{align*}
\IP{\partial_t\nu}{\tau}
 &= -\IP{\nu}{\partial_t\tau}
 = -\IP{\nu}{-k_{sss}\nu + (\partial^\top_t\!\gamma)\gamma_{ss}}
\\
 &= k_{sss} - k (\partial_t^\top\!\gamma).
\end{align*}
Returning to the area functional, we can now directly evaluate the derivative.
\begin{align*}
\rd{}{t}A
 &= -\frac{1}{2}\rd{}{t}\int_{\gamma} \IP{\gamma}{\nu} ds
\\
 &= -\frac{1}{2}\int_{\gamma} -k_{ss} + \IP{\gamma}{\nu_t} + \IP{\gamma}{\nu}kk_{ss} - (\partial^\top_t\!\gamma)\IP{\gamma}{\nu_s} ds
\\
 &= -\frac{1}{2}\int_{\gamma} \IP{\gamma}{k_{sss}\tau - k(\partial^\top_t\!\gamma)\tau} + \IP{\gamma}{\nu}kk_{ss} - (\partial^\top_t\!\gamma)\IP{\gamma}{-k\tau} ds
\\
 &= -\frac{1}{2}\int_{\gamma} -k_{ss} - \IP{\gamma}{\tau_s}k_{ss} + \IP{\gamma}{\nu}kk_{ss} ds
\\
 &= 0,
\end{align*}
where we used Lemma \ref{DI} with $f = \IP{\gamma}{\nu}$ in the second line and integration by parts, the periodicity of $\gamma$ and the formulae \eqref{SF} throughout.

\end{proof}

We now turn our attention to the scale-invariant quantity
\[
\Ko = L\int_{\gamma} \big(k-\kav\big)^2 ds,
\]
where
\[
\kav = \frac{1}{L}\int_{\gamma} k ds.
\]
Note that we have (and will continue to) suppressed the dependence of $\Ko$ and $L$ on $\gamma(\cdot,t)$.
When we must indicate the dependence of $\Ko$ and $L$ on $\gamma(\cdot,t)$, we shall use the notation $\Ko(t) =
\Ko\big(\gamma(\cdot,t)\big)$ and $L(t) = L\big(\gamma(\cdot,t)\big)$.

A fundamental observation is that Lemmas \ref{LA} and \ref{WI} together imply $\Ko\in L^1([0,T))$.

\begin{lem}
\label{FO}
Suppose $\gamma:\S^1\times[0,T)\rightarrow\R^2$ solves \eqref{CD}.  Then
\[
\vn{\Ko}_1 < L^4(0)/16\pi^2.
\]
\end{lem}
\begin{proof}
Applying Lemma \ref{WI} with $f=k-\kav$ and recalling Lemma \ref{LA} we have
\[
\Ko \le \frac{L^3}{4\pi^2}\vn{k_s}_2^2 = - \frac{1}{16\pi^2}\rd{}{t}L^4,
\]
so
\[
\int_0^t\Ko\, d\tau \le \frac{L^4(0)}{16\pi^2}.
\qedhere
\]
\end{proof}

The above lemma holds regardless of initial data, and appears to indicate that the quantity $\Ko$ is a natural `energy' for
the flow.

\begin{rmk}
A similar argument as above also shows that $\vn{k_s}^2_2 \in L^1([0,T))$ with the estimate $\vn{\vn{k_s}_2^2}_1 \le L(0)$.
Although we will not need this fact, it does suggest that $\vn{k_s}^2_2$ is another well-behaved quantity under the flow.
\end{rmk}

There exists an $\omega\in\R$ satisfying
\begin{equation}
\int_{\gamma} k\, ds\bigg|_{t=0} = 2\omega\pi.
\label{TC}
\end{equation}
In the case where the solution is a family of closed curves, $\omega$ is the winding number of $\gamma(\cdot,0)$.
Since the solution is a one-parameter family of smooth diffeomorphisms, and the winding number is a topological invariant, the winding number of the curves $\gamma(\cdot,t)$ remains constant.
This can also be directly proven as in the lemma below.

\begin{lem}
\label{WN}
Suppose $\gamma:\S^1\times[0,T)\rightarrow\R^2$ solves \eqref{CD} and
\[
\int_{\gamma} k\, ds\bigg|_{t=0} = 2\omega\pi.
\]
Then
\[
\int_{\gamma} k\, ds = 2\omega\pi.
\]
In particular, the average curvature increases in absolute value with velocity
\[
\rd{}{t}\kav = \frac{2\omega\pi}{L^2}\vn{k_s}_2^2.
\]
\end{lem}
\begin{proof}
Differentiating $|\nu|^2 = 1$ and $|\tau|^2 = 1$ gives that $\IP{\nu_t}{\tau_s} = 0$ and $\IP{\nu}{\tau_{ss}} = 0$.
Using this and \eqref{SF}, \eqref{ET}, we compute the evolution of the curvature as
\begin{align}
\pd{}{t}k
 &= \partial_t\IP{\nu}{\gamma_{ss}}
 = \IP{\partial_t\nu}{\gamma_{ss}} + \IP{\nu}{\partial_t\gamma_{ss}}
\notag\\
 &= \IP{\nu}{\partial_t\tau_s}
\notag\\
 &= \IP{\nu}{\partial_s\tau_t
           - kk_{ss}\tau_s
           - \big(\partial_s\partial_t^\top\!\gamma\big)\tau_s}
\notag\\
 &= -k^2k_{ss} - k\big(\partial_s\partial_t^\top\!\gamma\big)
    +\IP{\nu}{\partial_s\tau_t}
\notag\\
 &= -k^2k_{ss} 
    -\IP{\nu}{\partial_s(k_{sss}\nu)} + (\partial^\top_t\!\gamma)\IP{\nu}{\gamma_{sss}}
\notag\\ \label{EK}
 &= -k_{ssss}-k^2k_{ss} 
     + (\partial^\top_t\!\gamma)k_s.
\end{align}
Therefore, applying Lemma \ref{DI} with $f = k$ we have
\[
\rd{}{t}\int_{\gamma} k ds
 = - \int_{\gamma} k_{ssss} + k^2k_{ss} - k^2k_{ss} + (\partial^\top_t\!\gamma)k_s - (\partial^\top_t\!\gamma)k_s ds
 = 0,
\]
using integration by parts and the periodicity of $\gamma$.  This completes the proof.
\end{proof}

We now compute the evolution of $\Ko$.

\begin{lem}
\label{KoE}
Suppose $\gamma:\S^1\times[0,T)\rightarrow\R^2$ solves \eqref{CD}.  Then
\begin{align*}
\rd{}{t}\Ko &+ \Ko\frac{\vn{k_s}_2^2}{L} + 2L\vn{k_{ss}}_2^2
\\
 &= 3L\int_{\gamma} (k-\kav)^2k^2_s ds
   + 6\kav L \int_{\gamma} (k-\kav)k_s^2 ds
   + 2\kav^2 L \vn{k_s}^2_2.
\end{align*}
\end{lem}
\begin{proof}
This is a direct computation.
\begin{align*}
\rd{}{t}\Ko
 &= - \int_{\gamma}k_s^2 ds \int_{\gamma} (k-\kav)^2 ds
   + 2L\int_{\gamma} (k-\kav)\big( -k_{ssss}-k^2k_{ss} 
                                + (\partial^\top_t\!\gamma)k_s \big)ds
\\
 &\qquad  + L\int_{\gamma} kk_{ss}(k-\kav)^2 - 2k_s(k-\kav)(\partial^\top_t\!\gamma) ds
\\
 &= - \Ko\frac{\vn{k_s}_2^2}{L} - 2L\vn{k_{ss}}_2^2
    + 2L\int_{\gamma} k^2k_s^2 ds
    + 4L\int_{\gamma} k(k-\kav)k^2_{s} ds
\\
 &\qquad
    - L\int_{\gamma} (k-\kav)^2k_s^2 ds
    - 2L\int_{\gamma} k(k-\kav)k^2_s ds
\\
 &= - \Ko\frac{\vn{k_s}_2^2}{L} - 2L\vn{k_{ss}}_2^2
    + 4L\int_{\gamma} k^2k_s^2 ds
    - 2\kav L\int_{\gamma} kk^2_{s} ds
\\
 &\qquad
    - L\int_{\gamma} (k-\kav)^2k_s^2 ds.
\end{align*}
Rearranging, we have
\begin{align*}
\rd{}{t}\Ko &+ \Ko\frac{\vn{k_s}_2^2}{L} + 2L\vn{k_{ss}}_2^2
\\
 &=   4L\int_{\gamma} k^2k_s^2 ds
    - 2\kav L\int_{\gamma} (k-\kav)k^2_{s} ds
    - 2\kav^2 L\int_{\gamma} k^2_{s} ds
    - L\int_{\gamma} (k-\kav)^2k_s^2 ds
\\
 &=   3L\int_{\gamma} (k-\kav)^2k_s^2 ds
    + 6\kav L\int_{\gamma} (k-\kav)k^2_{s} ds
    + 2\kav^2 L\int_{\gamma} k^2_{s} ds.
\end{align*}
This proves the lemma.
\end{proof}

Although $\Ko$ is a priori controlled in $L^1$, we need much finer control on $\Ko$ before we can assert control on other curvature quantities and deduce global existence.  (Indeed, global existence is not true in the class of solutions given by Theorem \ref{STE}.)

While $\Ko$ is small, we do have the desired control.  The following proposition gives us a pointwise estimate.

\begin{prop}
\label{DE}
Suppose $\gamma:\S^1\times[0,T)\rightarrow\R^2$ solves \eqref{CD}.  If there exists a $T^*$ such that for $t\in[0,T^*)$ we have
\begin{align*}
\Ko(t) &\le \frac{4\pi + 24\pi^2\omega^2 - 8\pi\sqrt{3\pi}\sqrt{\omega^2+3\pi\omega^4}}{3} = 2K^*,
\intertext{then during this time the estimate}
\Ko &+ 8\omega^2\pi^2\log L
 + \int_0^{t}\Ko\frac{\vn{k_s}_2^2}{L}d\tau
 \le \Ko(0) + 8\omega^2\pi^2\log L(0)
\end{align*}
holds.
\end{prop}
\begin{proof}
Lemma \ref{WN}, Corollary \ref{LU} and H\"older's inequality implies
\[
3L\int_{\gamma} (k-\kav)^2k_s^2 ds
 \le \frac{3L}{2\pi}\Ko\vn{k_{ss}}_2^2,
\]
and
\[
6\kav L\int_{\gamma} (k-\kav)k_s^2 ds
 \le 6\omega L\sqrt{\Ko}\vn{k_{ss}}_2^2.
\]
Thus, by Lemma \ref{KoE}
\begin{align*}
\rd{}{t}\Ko &+ \Ko\frac{\vn{k_s}_2^2}{L} + \Big(2-\Ko\frac{3}{2\pi}-6\sqrt{\Ko}\omega\Big)L\vn{k_{ss}}_2^2
 \le 2\kav^2 L \vn{k_s}^2_2.
\end{align*}
Now from Lemmas \ref{LA} and \ref{WN} we know
\[
   2\kav^2 L \vn{k_s}^2_2
 = \frac{8\omega^2\pi^2}{L} \vn{k_s}^2_2
 = -8\omega^2\pi^2\rd{}{t}\log L.
\]
Using the smallness of $\Ko$ and integrating finishes the proof.
\end{proof}

It is clear that even with $\Ko$ initially small, the estimate given by Proposition \ref{DE} is useless if we can not also exert good control on the ratio $L(0)/L(t)$.
In particular, we require that it remains only slightly larger than one.
This is easily achieved in the case we are interested in, that of closed curves with $\omega = 1$,
by an application of the isoperimetric inequality and Lemma \ref{LA}.  Observe
\[
\frac{L(0)}{L(t)}
 \le \frac{L(0)}{\sqrt{4\pi A(t)}}
 =   \frac{L(0)}{\sqrt{4\pi A(0)}}
 = \sqrt{I(0)},
\]
and this can be made arbitrarily close to one.

\begin{prop}
\label{EE1}
Suppose $\gamma:\S^1\times[0,T)\rightarrow\R^2$ solves \eqref{CD} and satisfies \eqref{Ewind1}.
Then
\[
\Ko(0) < K^*,\quad\text{ and }\quad I(0) < \exp\Big(\frac{K^*}{8\pi^2}\Big),
\]
implies
\[
\Ko \le 2K^*
\]
for all $t\in[0,T)$.
\end{prop}
\begin{proof}
From Theorem \ref{STE} and the smallness assumption there exists a maximal $T^* > 0$ such that
$\Ko(t) \le 2K^*$ for $t\in [0,T^*)$.
Suppose $T^* < T$.
Applying Proposition \ref{DE}, Lemma \ref{LA}, the isoperimetric inequality and the smallness assumption we have
\[
\Ko \le \Ko(0) + 8\pi^2\log\sqrt I
    < \frac{3K^*}{2},
\]
for all $t\in[0,T^*)$.  Taking $t\rightarrow\delta$ we arrive at a contradiction, and so $\delta = T$.
This finishes the proof.
\end{proof}

\section{Global existence}

We shall first prove Theorem \ref{GE1}.
There are two parts to this theorem: long time existence ($T = \infty$) and convergence
($\gamma(\S^1)$ approaches a round circle exponentially fast).
Given Proposition \ref{EE1} and the blowup criterion from \cite{DKS02}, it is rather straightforward to conclude the first
part of Theorem \ref{GE1}.

\begin{cor}
\label{LTE}
Suppose $\gamma:\S^1\times[0,T)\rightarrow\R^2$ solves \eqref{CD} and satisfies the assumptions of Theorem \ref{GE1}.  Then $T = \infty$.
\end{cor}
\begin{proof}
Suppose $T<\infty$.  Then by \cite[Theorem 3.1]{DKS02}, $\vn{k}_2^2\rightarrow\infty$ as $t\rightarrow T$.  However
\[
\Ko
 = L\int_{\gamma}(k-\kav)^2ds
 = L\vn{k}_2^2 - L^2\kav^2
 = L\vn{k}_2^2 - 4\pi^2,
\]
and since $2\sqrt{\pi A(\gamma_0)} \le L \le L(\gamma_0)$, this means $\Ko\rightarrow\infty$ as $t\rightarrow T$.  This is in direct contradiction with Proposition \ref{EE1}.
\end{proof}

It remains to classify the limit.
First observe that if we can show $\Ko\rightarrow0$ then we will obtain the desired convergence result (in a weaker topology), as curves with constant curvature in the plane are either circles or straight lines.
Lemma \ref{FO} and Corollary \ref{LTE} imply $\Ko\in L^1([0,\infty))$, and since $\Ko\ge 0$, to conclude this a sufficient
condition is that $\Ko'$ is uniformly bounded.

\begin{prop}\label{DE1}
Suppose $\gamma:\S^1\times[0,T)\rightarrow\R^2$ solves \eqref{CD} and satisfies the assumptions of Theorem \ref{GE1}. Then there exists a constant $c_1\in[1,\infty)$ depending only on $\gamma_0$ such that
\[
\vn{k_s}_2^2 \le c_1.
\]
\end{prop}
\begin{proof}
Let us assume there exist $\delta_0 \ge 0, \delta_0 < \delta_1$ such that $\vn{k_s}_2^2 > 1$ for $t\in[\delta_0,\delta_1]$. (Note that $\delta_0=0$ is allowed.)
Outside of such intervals we may take $c_1=1$.
We apply Lemma \ref{DI} to $k_s^2$ to obtain
\[
\rd{}{t}\int_\gamma k_s^2ds
 = \int_\gamma (k_s^2)_t + kk_{ss}k_s^2 - (k_s^2)_s(\partial_t^\top\!\gamma) ds.
\]
Now the interchange formula \eqref{COM} and the evolution of the curvature \eqref{EK} imply
\begin{align*}
k_{st}
 &= k_{ts} - kk_{ss}k_s - (\partial_t^\top\!\gamma)_sk_s
\\
 &= (-k_{ssss} - k^2k_{ss} + (\partial_t^\top\!\gamma)k_s)_s - kk_{ss}k_s - (\partial_t^\top\!\gamma)_sk_s
\\
 &= -k_{sssss} - 2kk_sk_{ss} - k^2k_{sss} + (\partial_t^\top\!\gamma)_sk_s + (\partial_t^\top\!\gamma)k_{ss} - kk_{ss}k_s - (\partial_t^\top\!\gamma)_sk_s
\\
 &= -k_{sssss} - 2kk_sk_{ss} - k^2k_{sss} + (\partial_t^\top\!\gamma)k_{ss} - kk_{ss}k_s.
\end{align*}
Therefore
\begin{align}
\rd{}{t}\int_\gamma k_s^2ds
 &= \int_\gamma (k_s^2)_t + kk_{ss}k_s^2 ds
\notag\\
 &= -2\int_\gamma k_sk_{sssss} + 2kk_s^2k_{ss} + k^2k_sk_{sss} ds - \int_\gamma kk_s^2k_{ss} ds
\notag\\
 &= -2\vn{k_{sss}}_2^2 + \frac{5}{3}\int_\gamma k_s^4 ds - 2\int_\gamma k^2k_sk_{sss}  ds
\notag\\
\label{AR1}
 &= -2\vn{k_{sss}}_2^2 + \frac{5}{3}\int_\gamma k_s^4 ds + 2\int_\gamma k^2k_{ss}^2ds + 4\int_\gamma kk_s^2k_{ss}ds.
\end{align}
Since
\[
\int_\gamma k_s^4 ds = -3\int_\gamma kk_s^2k_{ss}ds
 \le \frac{1}{2}\int_\gamma k_s^4 ds + \frac{9}{2}\int_\gamma k^2k_{ss}^2ds
\]
we have
\[
 \frac{5}{3}\int_\gamma k_s^4 ds + 4\int_\gamma kk_s^2k_{ss}ds
\le
 27\int_\gamma k^2k_{ss}^2ds.
\]
Combining this with \eqref{AR1} gives
\begin{equation*}
\rd{}{t}\int_\gamma k_s^2ds
 \le -2\vn{k_{sss}}_2^2 + 27\int_\gamma k^2k_{ss}^2 ds.
\end{equation*}
Rewriting the second term, integrating by parts, and using Lemma \ref{WI} we have
\begin{align*}
\int_\gamma k^2k_{ss}^2 ds
 &= \int_\gamma (k-\kav)^2k_{ss}^2ds + 2\kav\int_\gamma kk_{ss}ds - \kav^2\int_\gamma k_{ss}^2 ds
\\ &\le 2\int_\gamma (k-\kav)^2k_{ss}^2ds + 2\kav^2\int_\gamma k_{ss}^2 ds
\\ &\le \frac{\Ko}{\pi}\vn{k_{sss}}_2^2 + 2\kav^2\vn{k_s}_2\vn{k_{sss}}_2
\\ &\le \frac{\Ko}{\pi}\vn{k_{sss}}_2^2  + \frac{1}{54}\vn{k_{sss}}_2^2 + 54\kav^4\vn{k_{s}}_2^2
\\ &\le \frac{\Ko}{\pi}\vn{k_{sss}}_2^2  + \frac{1}{54}\vn{k_{sss}}_2^2 + \frac{27\kav^4\sqrt{L\Ko}}{\pi}\vn{k_{sss}}_2
\\ &\le \vn{k_{sss}}_2^2\Big(\frac{1}{27} +\frac{\Ko}{\pi}\Big) + \frac{27^3\kav^8L\Ko}{2\pi^2}.
\end{align*}
Noting that $\Ko < \frac{\pi}{27}$ by Proposition \ref{EE1} (in fact $\Ko \lesssim 0.106$) we find
\begin{align*}
\rd{}{t}\int_\gamma k_s^2ds
 + \vn{k_{sss}}_2^2
 \le \frac{27^3\kav^8L\Ko}{2\pi^2}.
\end{align*}
Observe that Lemma \ref{FO}, the isoperimetric inequality, and $\vn{k_s}_2^2 \ge 1$ imply
\[
\int_{\delta_0}^t\frac{27^3\kav^8L\Ko}{2\pi^2\vn{k_s}_2^2}
 \le \int_{\delta_0}^t\frac{27^3\kav^8L\Ko}{2\pi^2}
 \le c\vn{\Ko}_1 \le c,
\]
where $c=c(\gamma_0)$ is a universal constant.
Lemmas \ref{WI}, \ref{LA} and Proposition \ref{EE1} combined with the above and integrating now gives
\begin{align*}
\int_\gamma k_s^2ds
 \le ce^{-\frac{16\pi^4}{L^4(0)}(t-\delta_0)},
\end{align*}
where $t\in[\delta_0,\delta_1]$.
\end{proof}

As Proposition \ref{DE1} gives us good control over a quantity so long as it is larger than one, we name it a \emph{dissipation} estimate.
With this in hand, convergence to a round circle now follows quite easily.

\begin{cor}\label{CoC}
Suppose $\gamma:\S^1\times[0,T)\rightarrow\R^2$ solves \eqref{CD} and satisfies the assumptions of
Theorem \ref{GE1}.  Then $\gamma(\S^1)$ approaches a round circle with radius $\sqrt{\frac{A(0)}{\pi}}$.
\end{cor}
\begin{proof}
Corollary \ref{LTE} implies $T=\infty$, and Lemma \ref{FO} gives $\Ko\in L^1([0,\infty))$, so if $\Ko'$ is bounded, we shall be able to directly conclude $\Ko\rightarrow0$ and obtain the desired statement.
Recall Lemma \ref{KoE} and estimate
\begin{align*}
\left|\rd{}{t}\Ko\right| 
 &\le 3L\int_{\gamma} (k-\kav)^2k^2_s ds
   + 6\kav L \int_{\gamma} |k-\kav|k_s^2 ds
   + 2\kav^2 L \vn{k_s}^2_2
\\
 &\le \frac{3L^2(0)}{2\pi}\vn{k_s}_2^4
   + 12\omega\pi\sqrt{\frac{L(0)}{2\pi}}\vn{k_s}_2^3
   + 6\omega\pi\sqrt{\omega\pi A(0)} \vn{k_s}^2_2
\\
 &\le c(\omega,L(0),A(0))(c_1 + c_1^2),
\end{align*}
where we used Proposition \ref{DE1}.
This finishes the proof.
\end{proof}

Combining the convergence result with a short computation allows us to estimate the measure of the set of times during which the curvature is not strictly positive.

\begin{proof}[Proof of Proposition \ref{PwtE}]
Rearranging $\gamma$ in time if necessary, we may assume that
\begin{align*}
k(\cdot,t) \not> 0 ,\qquad &\text{ for all }t\in[0,t_0)\\
k(\cdot,t) > 0,    \qquad &\text{ for all }t\in[t_0,\infty)
\end{align*}
where $t_0 > \Big(\frac{L(0)}{2\pi}\Big)^4-\Big(\frac{A(0)}{\pi}\Big)^2$, otherwise we have nothing to prove.  However in this case we have
\begin{align*}
\rd{}{t}L &= -\vn{k_s}_2^2 \le -\frac{\pi^2}{L^2}\vn{k}_2^2 \le -\frac{4\pi^4}{L^3},&\text{for}\ t\in[0,t_0),
\intertext{where we used the fact that $\gamma$ is closed and that the curvature has a zero.  This implies}
L^4(t) &\le -16t\pi^4 + L^4(0),& \text{for}\ t\in[0,t_0),
\end{align*}
and thus $L^4(t_0) < 16\pi^2A^2(0)$.
This is in contradiction with the isoperimetric inequality.
\end{proof}

In the spirit of classical interpolation, we now obtain exponential decay of the $L^2$ norm of $k_{ss}$.

\begin{prop}\label{DE2}
Suppose $\gamma:\S^1\times[0,T)\rightarrow\R^2$ solves \eqref{CD} and satisfies the assumptions of Theorem \ref{GE1}. Then there exists a constant $c_2\in[1,\infty)$ depending only on $\gamma_0$ such that
\[
\vn{k_{ss}}_2^2 \le c_2e^{-\frac{4\pi^4}{L^4(0)}t}.
\]
\end{prop}
\begin{proof}
We compute and estimate
\begin{align*}
\rd{}{t}\int_\gamma k_{ss}^2 ds
 &= -2\int_\gamma k_{ssss}^2 ds + 3\int_\gamma k^2k_{sss}^2ds + 2\int_\gamma k^2k_{ss}k_{ssss} ds
\\
 &\le -\int_\gamma k_{ssss}^2 ds + 3\vn{k}_\infty^2\int_\gamma k_{sss}^2ds + \vn{k}_\infty^4\int_\gamma k_{ss}^2 ds
\\
 &\le -\frac{1}{2}\int_\gamma k_{ssss}^2 ds + \frac{11}{4}\vn{k}_\infty^4\int_\gamma k_{ss}^2 ds
\\
 &\le -\frac{1}{4}\int_\gamma k_{ssss}^2 ds + \frac{121\vn{k}_\infty^8}{16L}\Ko.
\end{align*}
After a fixed time translation we have $\vn{k}_\infty^8 \le 2\Big(\frac{\omega\pi}{A(0)}\Big)^4 \le c$ and so
\begin{equation*}
\rd{}{t}\int_\gamma k_{ss}^2 ds
 \le -\frac{4\pi^4}{L^4(0)}\int_\gamma k_{ss}^2 ds
 +
 c\Ko,
\end{equation*}
where we also used Lemma \ref{WI}.
Since $\Ko\in L^1([0,\infty))$, we apply Gronwall's inequality to obtain
\begin{equation*}
\int_\gamma k_{ss}^2 ds
\le c_2e^{-\frac{4\pi^4}{L^4(0)}t},
\end{equation*}
as required.
\end{proof}

We finish by giving uniform estimates for the higher derivatives of curvature.  For this there are
two obvious approaches.  With the previous estimate in hand, we have in fact shown that
$\vn{k_s}_\infty \rightarrow 0$ along a subsequence of times $t_j\rightarrow\infty$.  This allows us
to employ classical methods to obtain the exponential convergence; see \cite{EG97} for example.

However, the desired decay estimates also follow by combining our work here with the powerful
interpolation inequalities in \cite{DKS02}.  Although not as transparent and elementary as our
analysis above, it is very efficient.

\begin{prop}\label{DE3}
Suppose $\gamma:\S^1\times[0,T)\rightarrow\R^2$ solves \eqref{CD} and satisfies the assumptions of
Theorem \ref{GE1}. Then for each $m\in\N$ there exist constants $c_m,c_m'\in(0,\infty)$ depending only on $\gamma_0$ such
that
\[
\vn{\partial^m_sk}_2^2 \le c_me^{-tc_m'},\quad\text{ and }\quad
\vn{\partial^{m}_sk}_\infty \le \sqrt{L(0)c_{m+1}}e^{-\frac{t}{2}c_{m+1}'}.
\]
\end{prop}
\begin{proof}
Recall equation $(3.2)$ in \cite{DKS02}:
\begin{equation}
\rd{}{t}\int_\gamma (\partial^m_sk)^2 ds
 + \int_\gamma (\partial^{m+2}_sk)^2 ds
\le 
 c\vn{k}_2^{4m+10},
\label{EQksint}
\end{equation}
where $c$ is a constant depending only on $m$.
Noting that Proposition \ref{DE1} and Corollary \ref{CoC} imply $\vn{k}_2^2$ is uniformly bounded, we combine
\eqref{EQksint} with the simple interpolation inequality
\[
\vn{\partial^{m}_sk}^2_2
\le c(\vn{\partial^{m+2}_sk}^2_2 + \vn{k}^2_2)
\]
to obtain
\[
\rd{}{t}\int_\gamma (\partial^m_sk)^2 ds
 + c\int_\gamma (\partial^{m}_sk)^2 ds
\le 
 c_0,
\]
where $c$ and $c_0$ are absolute constants depending only on $m$ and $m$, $\gamma_0$ respectively.
Therefore
\[
\int_\gamma (\partial^m_sk)^2 ds
\le 
 c_0e^{-ct},
\]
and the $L^\infty$ estimates follow immediately.
\end{proof}

The proof of Theorem \ref{GE1} is now complete.

\section{Controlling density with $\Ko$}

We first treat a self-intersection as a singularity and `pull' information from it in a manner analogous to the proof of Simon's monotonicity formula \cite{S93}.  See also Theorem 6 in \cite{LY82} and
the appendix of \cite{KS04}.

\begin{lem}
\label{Lmult}
Suppose $\gamma:\S^1\rightarrow\R^2$ is a smooth immersed curve with $x\in\gamma(\S^1)$.
Then
\[
8|\gamma^{-1}(x)| = \int_\gamma (k^2 - k_0^2)|\gamma|ds
\]
where
\[
k_0 = 2\Big|\frac{\IP{\gamma}{\nu}}{|\gamma|^2} + \frac{k}{2} \Big|.
\]
\end{lem}
\begin{proof}
The right hand side of the equality above is translation invariant, so we may assume without loss of generality that $x=(0,0)$ is the origin.
Let $\varepsilon > 0$ and consider the test function $\eta:\S^1\rightarrow\R$ defined by
\[
\eta(s) = \min \Big\{ \frac{1}{\varepsilon}, \frac{1}{|\gamma(s)|} \Big\}.
\]
Note that
\[
\int_\gamma \eta_s^2 ds \le \int_\gamma \frac{1}{|\gamma|^4} ds \le \frac{L}{\varepsilon^4} < \infty,
\]
and so
\begin{equation}
\label{TmultE1}
\int_\gamma \eta_s\IP{\gamma}{\gamma_s} ds = - \int_\gamma \eta \IP{\gamma}{\gamma_s}_s ds.
\end{equation}
Computing, we have
\begin{align*}
\int_\gamma \eta_s\IP{\gamma}{\gamma_s} ds
 &= -\int_{\gamma^{-1}(\{|\gamma|\ge\varepsilon\})} \frac{\IP{\gamma}{\gamma_s}^2}{|\gamma|^3} ds
\intertext{and}
- \int_\gamma \eta \IP{\gamma}{\gamma_s} ds
 &= -\int_{\gamma^{-1}(\{|\gamma|<\varepsilon\})} \frac{1+k\IP{\gamma}{\nu}}{\varepsilon} ds
   -\int_{\gamma^{-1}(\{|\gamma|\ge\varepsilon\})} \frac{1+k\IP{\gamma}{\nu}}{|\gamma|} ds.
\end{align*}
Inserting these into \eqref{TmultE1} and simplifying gives
\begin{align*}
\frac{L(|\gamma|<\varepsilon)}{\varepsilon} &+ \frac{1}{\varepsilon}\int_{\gamma^{-1}(\{|\gamma|<\varepsilon\})} k\IP{\gamma}{\nu}ds
 = \int_{\gamma^{-1}(\{|\gamma|\ge\varepsilon\})} \frac{\IP{\gamma}{\gamma_s}^2}{|\gamma|^3} - \frac{1+k\IP{\gamma}{\nu}}{|\gamma|} ds
\\
 &= \int_{\gamma^{-1}(\{|\gamma|\ge\varepsilon\})} -\frac{\IP{\gamma}{\nu}^2}{|\gamma|^3} - k\frac{\IP{\gamma}{\nu}}{|\gamma|} ds
\\
&= - \int_{\gamma^{-1}(\{|\gamma|\ge\varepsilon\})} \bigg(\frac{\IP{\gamma}{\nu}}{|\gamma|^2} + \frac{k}{2}\bigg)^2|\gamma| ds
   + \frac{1}{4}\int_{\gamma^{-1}(\{|\gamma|\ge\varepsilon\})} k^2|\gamma| ds.
\end{align*}
Since $k\IP{\gamma}{\nu} = \IP{\gamma}{\gamma_{ss}}$ we have 
\[
\frac{L(|\gamma|<\varepsilon)}{\varepsilon} +
 \frac{1}{\varepsilon}\int_{\gamma^{-1}(\{|\gamma|<\varepsilon\})} k\IP{\gamma}{\nu}ds
= \frac{\IP{\gamma}{\gamma_s}}{\varepsilon}\Big|_{\partial(|\gamma|<\varepsilon)}.
\]
Thus
\begin{equation}
\label{TmultE2}
\frac{4\IP{\gamma}{\gamma_s}}{\varepsilon}\Big|_{\partial(|\gamma|<\varepsilon)}
 =   \int_{\gamma^{-1}(\{|\gamma|\ge\varepsilon\})} (k^2-k_0^2)|\gamma| ds.
\end{equation}
Noting that $\gamma$ is $C^1$ in $(0,0)$ and taking $\varepsilon\rightarrow 0$ finishes the proof.
\end{proof}

\begin{rmk}
The above proof requires only that $\gamma\in W^{2,2}$, as then (using Corollary \ref{LU} for example) $\gamma\in C^1$ and the $L^2$ norm of $k$ and $k_0$ is well-defined.
\end{rmk}

We shall also need the following well-known inequality for the $i$-th elementary symmetric functions
\[
\Pi_i(l_i,\ldots,l_n) = \sum_{1 \le k_1 < \cdots < k_i \le n} l_{k_1}\cdots l_{k_n}.
\]
\begin{lem}[Newton's inequality]
\label{Lnewt}
Let $l\in\R^n$ be a vector of positive real numbers.  Then
\[
\frac{\Pi_{i+1}(l)}{\Pi_{i+2}(l)} \ge \frac{\Pi_i(l)}{\Pi_{i+1}(l)}\frac{{n \choose i+1}^2}{{n \choose i}{n \choose i+2}}.
\]
\end{lem}

\begin{proof}[Proof of Theorem \ref{Tmult}]
Since $\gamma$ is compact, there exists an $x_0\in\R^2$ such that
\[
m(\gamma) = \sup_{x\in\R^2} |\gamma^{-1}(x)| = |\gamma^{-1}(x_0)|.
\]
Applying Lemma \ref{Lmult} in $x_0$ we have
\begin{equation}
\label{TmultE3}
8m = 8|\gamma^{-1}(x_0)| = \int_\gamma (k^2 - k_0^2)|\gamma|ds \le \int_\gamma k^2|\gamma|ds.
\end{equation}
We now decompose $\gamma$ into $m$ closed arcs $\gamma_i,\ldots,\gamma_m$, each smooth outside of the point $x_0$.
Let $l_i$ denote the length of the arc $\gamma_i$.
Our goal is to apply \eqref{TmultE3} to each arc $\gamma_i$, however these curves are not regular enough at $x_0$.
To ameliorate this point, consider an associated curve $\tilde\gamma_i:[0,l_i)\rightarrow\R^2$, also with length $l_i$, which
is without self-intersections, smooth outside $x_0$ and satisfies
\begin{align*}
\tilde\gamma_i(0) &= \gamma_i(0) = x_0
\\
\lim_{\varepsilon\searrow0} \partial_s\tilde\gamma_i(\varepsilon)
&= \lim_{\varepsilon\searrow0} \partial_s\gamma_i(\varepsilon)
\\
\int_{\tilde\gamma_i}\tilde k^2|\tilde\gamma|ds &\le 
\int_{\gamma_i}k^2|\gamma_i|ds.
\end{align*}
This is realised for example by reflecting $\gamma_i$ across the line
\[
x_0 + r\ \text{rot}_{\pi/2}\Big(\partial_s\gamma_i(0) + \partial_s\gamma_i(l_i)\Big),\qquad r\in\R.
\]
We now consider the extension $\tilde\gamma_i$ of each arc $\gamma_i:[0,l_i)\rightarrow\R^2$ defined by
\[
\hat\gamma_i(s) =
 \begin{cases}
   \gamma_i(s)&\text{ for }s\in [0,l_i),
\\
   \tilde\gamma_i(s-l_i)&\text{ for }s\in [l_i,2l_i).
 \end{cases}
\]
Note that $m(\hat\gamma_i) = 2$, and $|\hat\gamma_i| \le l_i/4$.  Applying \eqref{TmultE3} to $\hat\gamma_i$ we have
\[
8
 \le \int_{\hat\gamma_i} \hat k^2|\hat \gamma_i|ds
 \le 2\int_{\gamma_i} k^2|\gamma_i|ds
 \le \frac{l_i}{2}\int_{\gamma_i} k^2 ds.
\]
Therefore
\begin{equation}
\label{TmultE4}
\int_\gamma k^2 ds = \sum_{i=1}^m \int_{\gamma_i} k^2 ds \ge 16\sum_{i=1}^m \frac{1}{l_i}
 = 16\frac{\Pi_{m-1}(l_1,\ldots,l_m)}{\Pi_m(l_1,\ldots,l_m)}.
\end{equation}
Iterating Lemma \ref{Lnewt} $(m-1)$-times over gives the estimate 
\[
\frac{\Pi_{m-1}(l_1,\ldots,l_m)}{\Pi_m(l_1,\ldots,l_m)}
\ge
\frac{m^2}{\Pi_1(l_1,\ldots,l_m)}.
\]
Since $\Pi_1(l_1,\ldots,l_m) = L$, combining this estimate with \eqref{TmultE4} implies
\[
L\int_\gamma k^2 ds
 = \Ko - 4\pi^2\omega^2
 \ge 16m^2 - 4\pi^2\omega^2,
\]
as required.
\end{proof}

\bibliographystyle{plain}
\bibliography{curvediffusionfinal}

\begin{thebibliography}{10}

\bibitem{A93}
H.~Amann.
\newblock {Nonhomogeneous linear and quasilinear elliptic and parabolic
  boundary value problems}.
\newblock {\em Function Spaces, Differential Operators and Nonlinear Analysis},
  pages 9--126, 1993.

\bibitem{Abook}
H.~Amann.
\newblock {\em {Linear and quasilinear parabolic problems. Vol. 1: Abstract
  linear theory}}.
\newblock 1995.

\bibitem{A05}
H.~Amann.
\newblock {Quasilinear parabolic problems via maximal regularity}.
\newblock {\em Adv. Differential Equations}, 10(10):1081--1110, 2005.

\bibitem{A90}
S.B. Angenent.
\newblock {Nonlinear analytic semiflows}.
\newblock {\em Proc. Roy. Soc. Edinburgh. Sect. A}, 115(1-2):91--107, 1990.

\bibitem{A10smoothing}
T.~Asai.
\newblock On smoothing effect for higher order curvature flow equations.
\newblock {\em Adv. Math. Sci. Appl.}, 20(2):483, 2010.

\bibitem{BDR84}
P.~Baras, J.~Duchon, and R.~Robert.
\newblock {Evolution d'une interface par diffusion de surface}.
\newblock {\em Comm. Partial Differential Equations}, 9(4):313--335, 1984.

\bibitem{BBW98}
Andrew~J. Bernoff, Andrea~L. Bertozzi, and Thomas~P. Witelski.
\newblock Axisymmetric surface diffusion: Dynamics and stability of
  self-similar pinch-off.
\newblock Technical report, J. Statist. Phys, 1998.

\bibitem{B10}
S.~Blatt.
\newblock Loss of convexity and embeddedness for geometric evolution equations
  of higher order.
\newblock {\em Journal of Evolution Equations}, 10(1):21--27, 2010.

\bibitem{CTNc96}
J.W. Cahn, C.M. Elliott, and A.~Novick-Cohen.
\newblock {The Cahn--Hilliard equation with a concentration dependent mobility:
  motion by minus the Laplacian of the mean curvature}.
\newblock {\em European J. Appl. Math.}, 7(03):287--301, 1996.

\bibitem{C91}
P.T. Chru{\'s}ciel.
\newblock {Semi-global existence and convergence of solutions of the
  Robinson-Trautman (2-dimensional Calabi) equation}.
\newblock {\em Comm. Math. Phys.}, 137(2):289--313, 1991.

\bibitem{DKS02}
G.~Dziuk, E.~Kuwert, and R.~Sch{\"a}tzle.
\newblock {Evolution of elastic curves in $\R^n$: existence and computation}.
\newblock {\em SIAM J. Math. Anal.}, 33(5):1228--1245, 2002.

\bibitem{Ebook}
S.D. Eidel'man.
\newblock {\em {Parabolic systems}}.
\newblock 1969.

\bibitem{EZbook}
S.D. Eidel'man and N.V. Zhitarashu.
\newblock {\em {Parabolic boundary value problems}}.
\newblock 1998.

\bibitem{EG97}
C.M. Elliott and H.~Garcke.
\newblock {Existence results for diffusive surface motion laws}.
\newblock {\em Adv. Math. Sci. Appl.}, 7(1):467--490, 1997.

\bibitem{EM01losing}
C.M. Elliott and S.~Maier-Paape.
\newblock Losing a graph with surface diffusion.
\newblock {\em Hokkaido Math. J.}, 30:297--305, 2001.

\bibitem{EI05}
J.~Escher and K.~Ito.
\newblock {Some dynamic properties of volume preserving curvature driven
  flows}.
\newblock {\em Math. Ann.}, 333(1):213--230, 2005.

\bibitem{EMS98}
J.~Escher, U.F. Mayer, and G.~Simonett.
\newblock {The surface diffusion flow for immersed hypersurfaces}.
\newblock {\em SIAM J. Math. Anal.}, 29(6):1419--1433, 1998.

\bibitem{EM10surface}
J.~Escher and P.B. Mucha.
\newblock The surface diffusion flow on rough phase spaces.
\newblock {\em Discrete and Contin. Dyn. Syst.}, 26(2):431--453, 2010.

\bibitem{ES99}
J.~Escher and G.~Simonett.
\newblock {Moving surfaces and abstract parabolic evolution equations}.
\newblock {\em Topics in nonlinear analysis: the Herbert Amann anniversary
  volume}, page 183, 1999.

\bibitem{FGG08}
A.~Ferrero, F.~Gazzola, and H.C. Grunau.
\newblock {Decay and eventual local positivity for biharmonic parabolic
  equations}.
\newblock {\em Dyn. Syst.}, 21(4):1129--1157, 2008.

\bibitem{Fbook}
A.~Friedman.
\newblock {\em {Partial differential equations of parabolic type}}, volume~38.
\newblock 1964.

\bibitem{GH86}
M.~Gage and R.S. Hamilton.
\newblock {The heat equation shrinking convex plane curves}.
\newblock {\em J. Differential Geom.}, 23(1):69--96, 1986.

\bibitem{GG08}
F.~Gazzola and H.C. Grunau.
\newblock {Eventual local positivity for a biharmonic heat equation in $\R^n$}.
\newblock {\em Discrete Contin. Dyn. Syst., Ser. S}, 1:83--87, 2008.

\bibitem{GG09}
F.~Gazzola and H.C. Grunau.
\newblock {Some new properties of biharmonic heat kernels}.
\newblock {\em Nonlinear Anal.}, 70(8):2965--2973, 2009.

\bibitem{GI98pinching}
Y.~Giga and K.~Ito.
\newblock On pinching of curves moved by surface diffusion.
\newblock {\em Comm. Appl. Anal.}, 2(3):393--406, 1998.

\bibitem{GI99loss}
Y.~Giga and K.~Ito.
\newblock Loss of convexity of simple closed curves moved by surface diffusion.
\newblock In {\em {Topics in Nonlinear Analysis, The Herbert Amann anniversary
  volume (eds. J. Escher and G. Simonett)}}, volume~35 of {\em {Progress in
  Nonlinear Differential Equations and Their Applications}}, pages 305--320.
  Birkh{\"a}user, 1999.

\bibitem{H82nash}
R.S. Hamilton.
\newblock {The inverse function theorem of Nash and Moser}.
\newblock {\em J. Amer. Math. Soc.}, 7(1), 1982.

\bibitem{H82}
R.S. Hamilton.
\newblock {Three-manifolds with positive Ricci curvature}.
\newblock {\em J. Differential Geom.}, 17:255--306, 1982.

\bibitem{HP96}
G.~Huisken and A.~Polden.
\newblock {Geometric evolution equations for hypersurfaces. Calc. of Var. and
  Geom. Evo. Probl., CIME Lectures of Cetraro}, 1996.

\bibitem{KL11}
H.~Koch and T.~Lamm.
\newblock Geometric flows with rough initial data.
\newblock {\em Arxiv preprint arXiv:0902.1488}, 2009.
\newblock To appear in Asian J. Math.

\bibitem{KS01}
E.~Kuwert and R.~Sch{\"a}tzle.
\newblock {The Willmore flow with small initial energy}.
\newblock {\em J. Differential Geom.}, 57(3):409--441, 2001.

\bibitem{KS02}
E.~Kuwert and R.~Sch{\"a}tzle.
\newblock {Gradient flow for the Willmore functional}.
\newblock {\em Comm. Anal. Geom.}, 10(2):307--339, 2002.

\bibitem{KS04}
E.~Kuwert and R.~Sch{\"a}tzle.
\newblock {Removability of point singularities of Willmore surfaces}.
\newblock {\em Ann. of Math.}, 160(1):315--357, 2004.

\bibitem{LY82}
P.~Li and S.T. Yau.
\newblock A new conformal invariant and its applications to the {W}illmore
  conjecture and the first eigenvalue of compact surfaces.
\newblock {\em Invent. Math.}, 69(2):269--291, 1982.

\bibitem{Lbook}
A.~Lunardi.
\newblock {\em {Analytic semigroups and optimal regularity in parabolic
  problems}}.
\newblock 1995.

\bibitem{P96}
A.~Polden.
\newblock {\em {Curves and surfaces of least total curvature and fourth-order
  flows}}.
\newblock PhD thesis, 1996.

\bibitem{S04}
J.J. Sharples.
\newblock {Linear and quasilinear parabolic equations in Sobolev space}.
\newblock {\em J. Differential Equations}, 202(1):111--142, 2004.

\bibitem{S93}
L.~Simon.
\newblock Existence of surfaces minimizing the {W}illmore functional.
\newblock {\em Comm. Anal. Geom}, 1(2):281--326, 1993.

\bibitem{CT94}
J.E. Taylor and J.W. Cahn.
\newblock {Linking anisotropic sharp and diffuse surface motion laws via
  gradient flows}.
\newblock {\em J. Stat. Phys.}, 77(1):183--197, 1994.

\bibitem{W09}
G.E. Wheeler.
\newblock {\em {Fourth order geometric evolution equations}}.
\newblock PhD thesis, 2009.

\bibitem{W10}
G.E. Wheeler.
\newblock {Surface diffusion flow near spheres}.
\newblock {\em To appear in Calc. Var. Partial Differential Equations}, 2010.

\end{thebibliography}

\end{document}